\documentclass[12pt]{amsart}
\usepackage[a4paper,margin=2.2cm]{geometry}
\usepackage{microtype}
\usepackage[cal=euler]{mathalfa}
\usepackage{tikz,tikz-cd}
\usetikzlibrary{decorations}
\usepackage[colorlinks,allcolors=blue]{hyperref}
\usepackage{graphicx}
\usepackage{xcolor}

\input{commands}
\defcite\and{andre_slope}
\defcite\dur{durov_new}
\defcite\kas{kashiwara_crystal}

\def\pa{proto-abelian\xspace}
\def\pe{proto-exact\xspace}
\def\qpa{parabelian\xspace}
\def\EM{Eilenberg-Moore category\xspace}

\def\NS{\opn{NS}} %normed spaces
\def\ES{\opn{ES}} %Euclidean (quadratic) spaces
\def\HS{\opn{HS}} %Hermitian spaces
\oper{CMon} %commutative monoids
\def\vect{\Vect^\fd}
\def\NN{\aleph_0}
\oper{Cl} %closure spaces
\oper{Mat} %matroids
\oper{Aff}
\oper{Norm}
\oper{Bun}
\oper{proj}
\oper{PCM} %partial commutative monoids
\opr\PC{\mathcal C} %pre-crystals
\operl{Ker,Coker,Im,Coim,coim,dom}

\def\aE{\mathfrak E} %adm epi
\def\aM{\mathfrak M} %adm mono
\def\Ex{\mathcal E} %adm exact pairs
\def\AA{\bar A} %monad of a semiring
\def\SS{\bar S} %monad of a partial semiring
\def\ZZ{\bZ_\infty} %monad

\opr\hom{\mathbf{Hom}} %internal Hom
\def\pto{\rightharpoonup} %partial function arrow
\def\darr{\rightrightarrows} %double arrow
\def\simm{\approx} %double \sim
\def\bm{\mathbf m}

\def\bplus{\mathrel{\btikz[scale=.24,line width=1.3,line cap=round]{%
  \draw(0,0)--(1,0)(.5,-.5)--(.5,.5);}}} %bold plus

\def\fd{\mathsf{fd}}
\def\s{\mathsf s} %strict morphisms

\operl{Kl,Ban,Lat,Hilb}
\begin{document}
\title{Proto-exact and parabelian categories}
\author{Sergey Mozgovoy}

\begin{abstract}
Proto-exact and parabelian categories serve as non-additive analogues of exact and quasi-abelian categories, respectively.
They give rise to algebraic $K$-theory and Hall algebras
similarly to the additive setting.
We show that every parabelian category admits a canonical proto-exact structure and we study several classes of parabelian categories,
including categories of normed and Euclidean vector spaces, pointed closure spaces and pointed matroids, Hermitian vector bundles over rings of integers.
We also examine finitary algebraic categories arising in Arakelov geometry and provide a criterion for determining when such a category is parabelian.
In particular, we prove that the categories of pointed convex spaces and absolutely convex spaces are parabelian.
\end{abstract}

\maketitle

\section*{Introduction}
Proto-exact categories \cite{dyckerhoff_highera} are a non-additive analogue of Quillen exact categories
\cite{quillen_higher,keller_chain}.
Essentially, a proto-exact category is a pointed category equipped with a class of short exact sequences (called admissible)
satisfying certain properties.
%closed under certain pullbacks and pushouts.
As in the additive setting, one can associate with a proto-exact category the Waldhausen $S$-construction,
%$\cS_\bul(\cC)$
which forms a 2-Segal simplicial groupoid.
This object serves as a foundation for constructing algebraic
$K$-theory as well as various types of Hall algebras associated with the category \cite{dyckerhoff_highera}.

Examples of (non-additive) proto-exact categories include
categories of quiver representations and semigroup representations in $\Set_*$
\cite{szczesny_representations,szczesny_hall},
category of Euclidean vector spaces
%(finite-dimensional)
with morphisms of norm $\le1$
and various categories arising in Arakelov geometry
\cite{andre_slope,dyckerhoff_highera},
category of pointed matroids \cite{eppolito_proto},
categories of modules over semirings \cite{jun_proto}.
%While some of the above proto-exact structures on the corresponding categories are defined ad hoc,
One can verify that all of the above proto-exact structures are canonical and consist of all short exact sequences.
This leads us to the notion of a parabelian
%(or quasi-proto-abelian)
category, which generalizes both quasi-abelian \cite{schneiders_quasi} and proto-abelian \cite{andre_slope,dyckerhoff_higher} categories.

Recall that an additive category with kernels and cokernels
is called quasi-abelian \cite{schneiders_quasi}
if the class of its strict epimorphisms (the cokernels) is closed under pullbacks
and the class of its strict monomorphisms (the kernels) is closed under pushouts.
Examples of (non-abelian) quasi-abelian categories include categories of filtered objects in an abelian category,
locally free sheaves over a smooth curve,
finitely generated free abelian groups (or finitely generated projective modules over a Dedekind domain).
A quasi-abelian category, when equipped with the class of all short exact sequences, forms an exact category, as shown in \cite{schneiders_quasi}.
In the non-additive setting, we define \qpa categories in a similar manner and prove that a \qpa category,
when equipped with the class of all short exact sequences, is \pe.
Note that the various definitions of \pa\ categories \cite{andre_slope,dyckerhoff_higher} are strictly stronger than that of a \qpa\ category.
For instance, the \pa\ categories in \cite{andre_slope} do not include pointed sets, while those in \cite{dyckerhoff_higher} exclude Euclidean vector spaces.

In this paper, we introduce a set of axioms for proto-exact categories that mirror those of exact categories \cite{quillen_higher, keller_chain}.
We further prove that this formulation is equivalent to the existing definition of a proto-exact category given in \cite{dyckerhoff_highera}.
Then, to establish that every \qpa category is \pe, it suffices to show that the classes of strict monomorphisms and epimorphisms are closed under composition.
This property is well known in the additive setting \cite{schneiders_quasi}.

We will verify that various classes of categories are \qpa.
For example, we will show that the categories of normed vector spaces
\S\ref{sec:normed}
and pointed closure spaces \S\ref{sec:closure} are \qpa.
The categories of Euclidean vector spaces and pointed matroids appear as their full subcategories closed under taking kernels and cokernels, hence are also \qpa.

We will also study a broad class of finitary algebraic categories,
defined as categories of algebras over finitary monads \S\ref{fini}.
Such categories naturally appear in universal algebra \cite{hyland_category}
and in the formal approach to Arakelov geometry \cite{durov_new}.
Although not all finitary algebraic categories are \qpa
(for instance, the category of groups fails to be \qpa), we will establish a criterion to determine when this is the case.

Given a normed semiring $K$, we will define the valuation monad $\bar\cO_K$ associated with the (weak) partial semiring $\cO_K=\sets{x\in K}{\n x\le1}$
and define the category of $\cO_K$-modules $\Mod(\cO_K)$ to be the category of algebras over the monad $\bar\cO_K$ \S\ref{val-mon}.
Such categories include categories of
\begin{enumerate}
\item
pointed sets (modules over $\bF_1=\set{0,1}\sbs\bN$),
\item monoid representations in pointed sets,
\item commutative monoids (modules over $\bN$),
\item modules over semirings,
\item convex spaces (modules over $[0,1]\sbs \bR_+$),
\item absolutely convex spaces (modules over $\bZ_{\infty}=[-1,1]\sbs\bR$).
\end{enumerate}
We will see that all of these categories are \qpa.
This naturally raises the question of characterizing those normed semirings $K$ for which the categories $\Mod(\cO_K)$ are \qpa.

\subsection*{Conventions}
We define $\bN=\bZ_{\ge0}=\sets{n\in\bZ}{n\ge0}$ and $\bR_+=\bR_{\ge0}=\sets{x\in\bR}{x\ge0}$.
For a subset $X\sbs Y$ and $y\in Y$, we denote $X\cup\set y$ by $X\cup y$ and $X\ms\set y$ by $X\ms y$.

\section{Proto-exact and \qpa categories}
In this section we will define proto-exact and \qpa categories
which can be thought of as non-additive analogues of exact and quasi-abelian categories, respectively.
As in the additive case, we will see that \qpa categories have a canonical proto-exact structure,
consisting of all short exact sequences.

\subsection{Pointed categories}
Given a cartesian square (or a pullback square)
\begin{equation}\label{sq1}
\begin{tikzcd}
X'\rar["f'"]\dar["g'"']&Y'\dar["g"]\\
X\rar["f"]&Y
\end{tikzcd}
\end{equation}
we say that $f'$ is a pullback of $f$ along $g$.
Similarly, if \eqref{sq1} is a cocartesian square (or a pushout square), we say that $f$ is a pushout of $f'$ along $g'$.
We say that \eqref{sq1} is bicartesian if it is both cartesian and cocartesian.
We say that a class of morphisms $\cA$ is closed under pullbacks along a class of morphisms $\cB$ if for all $f:X\to Y$ in $\cA$ and all $g:Y'\to Y$ in $\cB$ there exists a cartesian square \eqref{sq1} with $f'\in\cA$.
Similarly for pushouts of $\cA$ along $\cB$.

A category $\cC$ is called \idef{pointed} if it has a zero object $0$,
meaning an object that is both initial and terminal.
For $X,Y\in\cC$, we denote the composition
$X\to 0\to Y$ by~$0=0_{XY}$.
The kernel $\ker(f):\Ker(f)\to X$ of $f:X\to Y$ is the equalizer of $f,0:X\darr Y$
or the pullback of $0\to Y$ along $f$.
Similarly, the cokernel $\coker(f):Y\to\Coker(f)$ of $f:X\to Y$ is the coequalizer
of $f,0:X\darr Y$ or the pushout of $X\to0$ along $f$.
A pair $(f,g)$
of composable morphisms $X\xto fY\xto gZ$
is called \idef{exact} (or a short exact sequence) if $f=\ker g$ and $g=\coker f$.
In this case we say that $f$ is a \idef{strict mono(morphism)} and $g$ is a \idef{strict epi(morphism)}.

\begin{lemma}\label{lm cart}
Consider a commutative diagram
\begin{equation}
\begin{tikzcd}\label{diag1}
X'\rar["f'"]\dar["g'"'] &Y'\dar["g"]\\
X\rar["f"] &Y\rar["h"]&Z
\end{tikzcd}
\end{equation}
\begin{enumerate}
\item If $f=\ker(h)$ and the square is cartesian,
then $f'=\ker(hg)$.
\item If $f$ is mono and $f'=\ker(hg)$,
then the square is cartesian.
\end{enumerate}

\end{lemma}
\begin{proof}
\clm1
Given $a:U\to Y'$ satisfying $hga=0$,
there exists $b:U\to X$ with $fb=ga$.
There exists $c:U\to X'$ with $g'c=b$ and $f'c=a$.
If $f'c=f'd$, then $fg'c=fg'd$, hence $g'c=g'd$. Therefore $c=d$.

\clm2
Given $a:U\to Y'$ and $b:U\to X$ satisfying $ga=fb$,
we have $hga=hfb=0$, hence there exists $c:U\to X'$ with $f'c=a$.
We have $fg'c=gf'c=ga=fb$, hence $g'c=b$.
The map $c$ is unique since $f'$ is mono.
\end{proof}

\begin{lemma}[\cf \and]
\label{lm bicart}
A cocartesian square \eqref{sq1} with strict mono $f'$ and mono $f$ is bicartesian.
A cartesian square \eqref{sq1} with strict epi $f$ and epi $f'$ is bicartesian.
\end{lemma}
\begin{proof}
Let \eqref{sq1} be cocartesian, $f'$ be strict mono and $f$ be mono.
Let $h'=\coker(f'):Y'\to Z$.
There exists $h:Y\to Z$ with $hg=h'$ and $hf=0$.
\[\begin{tikzcd}
X'\rar["f'"]\dar["g'"'] &Y'\dar["g"]\drar["h'"]\\
X\rar["f"] &Y\rar["h"]&Z
\end{tikzcd}\]
Since $f$ is mono and $f'=\ker(hg)$, the square is cartesian
by Lemma \ref{lm cart}.
\end{proof}

We say that a pointed category $\cC$ has kernels and cokernels if every morphism in $\cC$ has a kernel and a cokernel.
For a morphism $f:X\to Y$ the coimage and the image are
\[\coim(f)=\coker(\ker f):X\to\Coim(f),\qquad \im(f)=\ker(\coker f):\Im(f)\to Y.\]
There is a unique factorization $X\to\Coim(f)\xto{\bar f}\Im(f)\to Y$
of $f$ and the morphism $f$ is called \idef{strict} if $\bar f$ is an isomorphism.
For a morphism $f$ the following conditions are equivalent
\begin{enumerate}
\item $f$ is a strict monomorphism.
\item $f$ is the kernel of some morphism.
\item $f$ is strict and is a monomorphism.
\end{enumerate}
Similarly for epimorphisms.

\begin{lemma}\label{pull/push}
If $\cC$ has kernels and cokernels,
then strict mono are closed under pullbacks along~$\cC$
and strict epi are closed under pushouts along~$\cC$.
\end{lemma}
\begin{proof}
For a strict mono $f:X\to Y$ and a morphism $g:Y'\to Y$,
let $h=\coker(f):Y\to Z$ and $f'=\ker(hg):X'\to Y'$ \eqref{diag1}.
We have $hgf'=0$ and $f=\ker(h)$, hence there exists $g':X'\to X$ with $fg'=gf'$.
The corresponding square is cartesian by Lemma~\ref{lm cart}.
\end{proof}

\begin{remark}
A functor between categories with kernels and cokernels
is called left (resp.\ right) exact if it preserves kernels
(resp.\ cokernels).
A functor is called exact if it is both left and right exact.
%and weakly exact if it preserves short exacts sequences.
By Lemma \ref{pull/push},
an exact functor also preserves pullbacks of strict monomorphisms and pushouts of strict epimorphisms.
\end{remark}

\subsection{Proto-exact categories}
\begin{definition}\label{def ex}
A proto-exact category is a pointed category $\cC$
equipped with a class $\Ex$ of exact pairs closed under isomorphism
such that the classes $\aM=\sets{f}{(f,g)\in\Ex}$,
$\aE=\sets{g}{(f,g)\in\Ex}$
of \idef{admissible} mono(morphisms) and epi(morphisms) satisfy
\begin{enumerate}
\item $1_X\in\aM\cap\aE$ for all $X\in\cC$.
\item $\aM$ and $\aE$ are closed under composition.
\item $\aE$ is closed under pullbacks along $\aM$.
\item $\aM$ is closed under pushouts along $\aE$.
\end{enumerate}
\end{definition}

The above pullbacks and pushouts are automatically bicartesian
by Lemma \ref{lm bicart}.

\begin{lemma}\label{pull existence1}
If the classes $\aM,\aE$ in Definition \ref{def ex}
%are closed under composition,
satisfy axiom (2),
then $\aM$ is closed under pullbacks along $\aE$ and $\aE$ is closed under pushouts along $\aM$.
\end{lemma}
\begin{proof}
For $f:X\to Y$ in $\aM$ and $g:Y'\to Y$ in $\aE$,
let $h=\coker(f):Y\to Z$.
Then $h\in\aE$, hence $hg\in\aE$
and $f'=\ker(hg):X'\to Y'$ is in $\aM$ \eqref{diag1}.
We have $hgf'=0$ and $f=\ker(h)$, hence there exists $g':X'\to X$ with $fg'=gf'$.
The corresponding square is cartesian by Lemma~\ref{lm cart}.
\end{proof}

\begin{remark}
The above results imply that the classes $\aM,\aE$ of morphisms in $\cC$ satisfy
\begin{enumerate}
\item $(0\to X)\in\aM$ and $(X\to0)\in\aE$ for all $X\in\cC$.
\item $\aM$ and $\aE$ are closed under composition
and contain all isomorphisms.
\item
For $f:X\to Y$ in $\aM$ and $g:Y'\to Y$ in $\aE$ there exists a bicartesian square \eqref{sq1} with $f'\in\aM$ and $g'\in\aE$.
\item For $f':X'\to Y'$ in $\aM$ and $g':X'\to X$ in $\aE$ there exists a bicartesian square \eqref{sq1} with $f\in\aM$ and $g\in\aE$.
\end{enumerate}
These are the axioms of a proto-exact category from \cite[\S2.4]{dyckerhoff_higher}.
Conversely, these axioms imply that if $f\in\aE$, then $\ker f\in\aM$ and $(\ker f,f)$ is an exact pair
and if $f\in\aM$, then $\coker f\in\aE$ and $(f,\coker f)$ is an exact pair.
Therefore the axioms of Definition \ref{def ex} are satisfied by
the class $\Ex$ of exact pairs $(f,g)$ for $f\in\aM$ and $g\in\aE$.
\end{remark}

\begin{remark}[Relation to exact categories]
An exact category \cite{quillen_higher,keller_chain} is an additive category~$\cC$ equipped with a class $\Ex$ of exact pairs closed under isomorphism
such that the classes $\aM,\aE$ defined as before satisfy
\begin{enumerate}
\item $1_0\in\aE$
\item $\aE$ is closed under composition.
\item $\aE$ is closed under pullbacks along all morphisms.
\item $\aM$ is closed under pushouts along all morphisms.
\end{enumerate}
One can show that
$1_X\in\aM\cap\aE$ for all $X\in\cC$
and that
$\aM$ is closed under composition
(see~\cite{keller_chain}).
This implies that an exact category is additive and proto-exact,
although the converse is not necessarily true.
\end{remark}

\subsection{Parabelian categories}

\begin{definition}
A category with kernels and cokernels is called \qpa if
\begin{enumerate}
\item Strict epimorphisms are closed under pullbacks along strict monomorphisms.
\item Strict monomorphisms are closed under pushouts along strict epimorphisms.
%\item The pullback of a strict epi along a strict mono exists and is a strict epi.
%\item The pushout of a strict mono along a strict epi exists and is a strict mono.
\end{enumerate}
\end{definition}

Note that the above pullbacks and pushouts exist automatically by Lemma \ref{pull/push}.

\begin{theorem}[\cf {\cite{andre_slope,schneiders_quasi}}]
\label{qpa is pe}
A \qpa category has a canonical proto-exact structure
consisting of all short exact sequences.
%where admissible epi (mono) are all strict epi (mono).
\end{theorem}
\begin{proof}
We need to show that given strict epimorphisms $Y'\xto{g}Y\xto{h}Z$,
the composition $hg$ is a strict epimorphism.
By assumption, the pullback of $g$ along $f=\ker(h)$ is a strict epimorphism.
\[\begin{tikzcd}
%&W\dar["e"]\dlar["e'"']\\
X'\rar["f'"]\dar["g'"'] &Y'\dar["g"]\\
X\rar["f"] &Y\rar["h"]&Z
\end{tikzcd}\]
We claim that $(f',hg)$ is an exact pair.
We have $f'=\ker(hg)$ by Lemma \ref{lm cart}.
Morphisms $g$ and $g'$ are strict epi, hence
the above cartesian square is bicartesian by Lemma \ref{lm bicart}.
Since $h=\coker(f)$ is the pushout of $X\to0$ along $f$, we conclude that $hg$ is the pushout of $X'\to0$ along $f'$, hence $hg=\coker(f')$.
\end{proof}

\begin{remark}
One can define a \qpa category and prove the above theorem by requiring only that the category is pointed and satisfies axioms (1), (2).
We included the requirement that $\cC$ has kernels and cokernels since it is satisfied in all our examples.
\end{remark}

\begin{remark}[Relation to quasi-abelian categories]
A quasi-abelian category \cite{schneiders_quasi}
%kashiwara_equivariant
is an additive category with kernels and cokernels such that
\begin{enumerate}
\item Strict epimorphisms are closed under pullbacks along all morphisms.
\item Strict monomorphisms are closed under pushouts along all morphisms.
\end{enumerate}
Therefore every quasi-abelian category is \qpa.
However, an additive \qpa doesn't have to be quasi-abelian.
\end{remark}

\begin{remark}[Relation to proto-abelian categories]
There are several distinct notions of \pa categories
in the literature \cite{andre_slope,dyckerhoff_higher}.
A proto-abelian category defined by Andr\'e \cite{andre_slope} is a \qpa category with the additional requirement that a morphism $f$ is mono \iff $\ker(f)=0$ and epi \iff $\coker(f)=0$.
For example, the category $\Set_*$ of pointed sets (see \S\ref{sec-pointed}) is not proto-abelian in this sense, since there exist non-monic morphisms $f$ in $\Set_*$ with $\Ker(f)=f\inv(0)=\set{0}$.
A proto-abelian category defined by Dyckerhoff \cite{dyckerhoff_higher} is a \qpa category with the additional requirement that all mono and epimorphisms are strict (this formulation is equivalent to the one in \cite{dyckerhoff_higher} if the category has kernels and cokernels).
For example, if $\cC$ is a \qpa category such that the class of its strict morphisms is closed under composition (such as $\cC=\Set_*$), then its subcategory $\cC^\s$ consisting of all strict morphisms is proto-abelian in this sense.
However, in general, the class of strict morphisms in a \qpa category need not be closed under composition (for example, this happens in \qpa categories of normed and Euclidean spaces \S\ref{sec:normed}).
\end{remark}

\begin{lemma}\label{subcat}
Let $F:\cB\to\cC$ be an exact functor reflecting isomorphisms between categories
with kernels and cokernels.
If $\cC$ is \qpa, then so is $\cB$.
\end{lemma}
\begin{proof}
By Lemma \ref{pull/push},
the pullbacks of strict mono and pushouts of strict epi are constructed by taking kernels and cokernels.
Therefore $F$ preserves them.
We just need to show that $F$ reflects strict mono and epi.
Let $f:X\to Y$ be a morphism in $\cB$ such that $Ff$ is a strict mono.
For $g=\coker(f)$ and $h=\ker(g):X'\to Y$, there exists a unique $a:X\to X'$ with $ha=f$.
\[\begin{tikzcd}
X\rar["f"]\dar["a"']&Y\rar["g"]&Z\\
X'\urar["h"']
\end{tikzcd}\]
We have $Fg=\coker(Ff)$ and $Fh=\ker(Fg)$, hence $Fa$ is an isomorphism since $Ff$ is a strict mono.
Therefore $a$ is an isomorphism.
\end{proof}

\begin{corollary}\label{parab subcat}
Let $\cC$ be a \qpa category and $\cB\sbs\cC$ be a full subcategory closed under taking kernels and cokernels.
Then $\cB$ is \qpa.
\end{corollary}

\begin{lemma}\label{T-alg}
Let $T$ be a monad on a parabelian category $\cC$ such that $T:\cC\to\cC$ is right exact (preserves cokernels).
Then the Eilenberg-Moore category $\cC^T$ is parabelian.
\end{lemma}
\begin{proof}
Since $T$ is right exact, we have $T0=0$, hence
$\cB=\cC^T$ has the zero object.
The category~$\cB$ has kernels and cokernels.
Indeed, for $g:Y\to Z$ in $\cB$, consider $f=\ker(g):X\to Y$ in $\cC$.
Then there exists a unique morphism $\al_X:TX\to X$ that makes the following diagram commute
\[\begin{tikzcd}
TX\rar["Tf"]\dar["\al_X"]&TY\rar["Tg"]\dar["\al_Y"]&TZ\dar["\al_Z"]\\
X\rar["f"]&Y\rar["g"]&Z
\end{tikzcd}\]
This equips $X$ with the structure of a $T$-algebra
so that $f$ becomes the kernel of $g$ in $\cB$.
On the other hand, for $f:X\to Y$ in $\cB$, let $g=\coker(f):Y\to Z$ in $\cC$.
Then $Tg=\coker(Tf)$,
hence there exists a unique morphism $\al_Z:TZ\to Z$ that makes the above diagram commute.
This equips $Z$ with the structure of a $T$-algebra
so that $g$ becomes the cokernel of $f$ in $\cB$.
This implies that the forgetful functor $F:\cB=\cC^T\to\cC$ is exact.
It also reflects isomorphisms and we conclude that $\cB$ is parabelian by Lemma~\ref{subcat}.
\end{proof}

Let $\cC$ be a symmetric monoidal category and $A$ be a monoid in $\cC$.
We define the category $\Mod(A)=\Mod(A,\cC)$ of (left) $A$-modules to be the category $\cC^T$ for the monad $T:\cC\to\cC$, $X\mto A\ts X$.

\begin{corollary}
\label{parab-modules1}
Let $\cC$ be a closed symmetric monoidal category and $A$ be a monoid in $\cC$.
If $\cC$ is \qpa, then $\Mod(A)$ is \qpa.
\end{corollary}
\begin{proof}
The functor $T=A\ts\blank:\cC\to\cC$ is right exact as it is left adjoint to the functor $\hom(A,\blank)$, where $\hom$ is the internal \Hom.
\end{proof}

\begin{lemma}\label{Fun}
Let $\cC$ be a \qpa category and $\cI$ be a small category.
Then the category of functors $[\cI,\cC]$ is \qpa.
\end{lemma}
\begin{proof}
The kernels and cokernels in $[\cI,\cC]$ are constructed componentwise.
Therefore a morphism is strict mono/epi \iff all of its components are strict mono/epi.
By Lemma \ref{pull/push} the pullbacks of strict monomorphisms and the pushouts of strict epimorphisms are also defined componentwise.
This implies that $[\cI,\cC]$ is \qpa.
\end{proof}

Let $Q=(Q_0,Q_1,s,t)$ be a quiver (directed graph)
and $\Psi$ be a $Q$-diagram of categories and functors,
meaning a collection of categories~$\Psi_i$ for $i\in Q_0$
and functors $\Psi_a:\Psi_{i}\to\Psi_{j}$ for arrows $a:i\to j$ in $Q$.
Let $\Rep(\Psi)$ be the category with objects $X$
that are collections of objects $X_i\in\cA_i$ for all $i\in Q_0$ and
morphisms $X_a:\Psi_aX_i\to X_j$ for all arrows $a:i\to j$ in $Q$ (\cf \cite{mozgovoy_quiver}).
A morphism $f:X\to Y$ in $\Rep(\Psi)$ is a collection of morphisms $f_i:X_i\to Y_i$ for $i\in Q_0$ such that $f_j X_a=Y_a\Psi_a(f_i)$ for all arrows $a:i\to j$.
Let $\Rep^\bi(\Psi)\sbs\Rep(\Psi)$ be the full subcategory consisting of objects $X$ such that the morphisms $X_a$ are isomorphisms for all arrows $a$.

\begin{theorem}\label{Q-diagram}
If $\Psi$ is a $Q$-diagram such that the categories $\Psi_i$ for $i\in Q_0$ are parabelian and the functors $\Psi_a$ for $a\in Q_1$ are right exact, then the category $\Rep(\Psi)$ is \qpa.
If the functors $\Psi_a$ for $a\in Q_1$ are exact,
then the category $\Rep^\bi(\Psi)$ is \qpa.
\end{theorem}
\begin{proof}
The category $\Rep(\Psi)$ has kernels and cokernels defined componentwise by \cite{mozgovoy_quiver}.
Therefore a morphism in $\Rep(\Psi)$ is strict mono/epi \iff all of its components are strict mono/epi.
By Lemma \ref{pull/push} the pullbacks of strict monomorphisms and the pushouts of strict epimorphisms are also defined componentwise.
This implies that $\Rep(\Psi)$ is \qpa.

Assuming that the functors $\Psi_a$ for $a\in Q_1$ are exact,
let us show that $\Rep^\bi(\Psi)\sbs\Rep(\Psi)$ is closed under taking kernels and cokernels.
Let $f:X\to Y$ be a morphism in $\Rep^\bi(\Psi)$ and $K=\Ker(f)$ in $\Rep(\Psi)$.
For every arrow $a:i\to j$ we have a commutative diagram
\[\begin{tikzcd}
\Psi_aK_i\rar["\Psi_ah_i"]\dar["K_a"]&\Psi_a X_i\rar["\Psi_af_i"]\dar["X_a"]&\Psi_a Y_i\dar["Y_a"]\\
K_j\rar["h_j"]&X_j\rar["f_j"]&Y_j
\end{tikzcd}\]
where $h_i=\ker(f_i)$ and $\Psi_a h_i=\ker(\Psi_a f_i)$ (since $\Psi$ is exact).
The morphisms $X_a$ and $Y_a$ are isomorphisms, hence $K_a$ is an isomorphism.
This implies that $K\in\Rep^\bi(\Psi)$.
The proof for cokernels is the same.
The category $\Rep^\bi(\Psi)$ is \qpa by Corollary \ref{parab subcat}.
\end{proof}

\section{First examples}

\subsection{Pointed sets}
\label{sec-pointed}
Let $\Set_*$ be the category of pointed sets $(X,0_X)$, where $0=0_X\in X$ is called the \idef{base point}.
This category has all limits and colimits and is parabelian,
with the zero object $\set0$ (also denoted by $0$).
The forgetful functor $U:\Set_*\to\Set$ has the left adjoint \[F:\Set\to\Set_*,\qquad X\mto X_*=(X\sqcup 0_X,0_X).\]
Using the monad $\bF_1=UF:\Set\to\Set$,
we get an equivalence $\bar U:\Set_*\to\Set^{\bF_1}$, $X\mto UX$,
to the \EM $\Mod(\bF_1)=\Set^{\bF_1}$ of $\bF_1$-algebras
(usually called $\bF_1$-modules in this context).

For $\es\ne A\sbs X$, let $X/A\iso (X\ms A)_*$ be the quotient of $X$ by the equivalence relation generated by $x\sim y$ for $x,y\in A$.
We equip $X/A$ with the base point $A$.
For a morphism $f:X\to Y$ in $\Set_*$ we have
\[\Ker(f)=(f\inv(0_Y),0_X),\qquad
\Coker(f)=Y/f(X),\]
\[\Coim(f)=X/f\inv(0_Y)\iso\dom(f)_*,\qquad
\Im(f)=(f(X),0_Y),\]
where $\dom(f)=X\ms f\inv(0_Y)$.
Therefore $f$ is strict \iff $f$ is injective on $\dom(f)$.
This implies that the class of strict morphisms is closed under composition.
Therefore we also have the parabelian subcategory $\Set_*^\s\sbs \Set_*$ consisting of strict morphisms (sometimes called the category of vector spaces over $\bF_1$).

\begin{remark}[Partial functions]
A morphism $f:X_*\to Y_*$ in $\Set_*$ can be identified with a \idef{partial function}
$f:X\pto Y$
(defined on the subset $\dom(f)=X\ms f\inv(0_Y)$,
%\href{https://en.wikipedia.org/wiki/Partial_function}{wiki}.
called the \idef{domain of definition} of $f$).
%or natural domain) is $\dom(f)=X\ms f\inv(0)$.
A partial function $f$ is called a \idef{partial bijection} if $f$ is injective on $\dom(f)$ (we call such maps strict).
Partial bijections $f:X\pto X$ form a monoid
$\mathcal{IS}_X\iso\End_{\Set_*^\s}(X_*)$,
called a \idef{symmetric inverse semigroup}.
%\href{https://en.wikipedia.org/wiki/Inverse_semigroup}{wiki}.
A partial function $f$ is called \idef{total} if $\dom(f)=X$,
meaning that $\ker f=0$.
\end{remark}

The products and coproduct in $\Set_*$ are defined by
\[(X\xx Y,(0_X,0_Y)),\qquad X\oplus Y=(X\sqcup Y)/\set{0_X, 0_Y}.\]
The category $\Set_*$ has the structure of a closed symmetric monoidal category (\cf Remark \ref{com-monads}) with the tensor product
$X\m Y=(X\xx Y)/(X\xx\set{0_Y}\cup \set{0_X}\xx Y)$
(called the \idef{smash product})
and the internal hom-object $\hom(X,Y)=\Hom(X,Y)$ equipped with the base point $0_{XY}$.
The identity object for this tensor product is $\one=\bF_1\set1=\set{0,1}$.
There is a duality functor (\cf \cite{young_degenerate})
\[D:(\Set_*^\s)^\op\to\Set_*^\s,\qquad X\mto\Hom_{\Set_*^\s}(X,\one).\]
For every $\vi\in\Hom_{\Set_*^\s}(X,\one)$,
either $\vi=0$ or there exists a unique $x\in X\ms0_X$ with $\vi\inv(1)=\set x$.
Therefore we can identify $DX$ with $X$.
For a strict morphism $f:X\to Y$, the dual $f^*=Df$ is given by
\begin{equation}\label{dual}
f^*:Y\to X,\qquad
f^*(y)=
\begin{cases}
f\inv(y)& y\in f(X)\ms 0_Y,\\
0_X& \text{otherwise}.
\end{cases}
\end{equation}
This implies that $D^2\iso\id$.

\subsubsection{Modules over monoids}
\label{mod over monoids}
A monoid $A$ in $(\Set_*,\m)$ is called a pointed monoid.
It can be identified with a usual monoid that has an absorbing element (element $0\in A$ such that $0x=x0=0$ for all $x\in A$).
Given a pointed monoid $A$, we consider the monad $T=A\m-:\Set_*\to\Set_*$ and define the category of left $A$-modules (or $A$-sets) $\Mod(A)=\Set_*^T$ (\cf \cite{chu_sheaves,szczesny_hall}).
Since $T$ is left adjoint to $\hom(A,-)$,
it preserves colimits, hence $\Mod(A)$ has all limits and colimits \S\ref{fini}.
In particular, $T$ is right exact,
hence $\Mod(A)$ is (strictly) \qpa by Lemma \ref{T-alg}.
For a usual monoid $G$, consider the pointed monoid $G_*=G\sqcup 0_G$
with the absorbing element $0_G$.
The category $\Set_*^G=\Mod(G_*)$ consists of pointed sets equipped with a $G$-action and can be identified with the category of functors $[G,\Set_*]$, where $G$ is interpreted as a category with one object and the endomorphism monoid $G$.
The category $[G,\Set_*^\s]$ of strict $G$-modules (called modules of type \al in \cite{szczesny_hopf})
is also parabelian by Lemma \ref{Fun}.

\subsubsection{Quiver representations}
Let $\vect_{\bF_1}\sbs\Set_*^\s$ be the full subcategory of finite pointed sets,
interpreted as the category of finite-dimensional vector spaces over $\bF_1$.
% (\cf~\cite{szczesny_representations}).
Given a quiver~$Q$, let $\cP(Q)$ be its path category (with the vertices of $Q$ as objects and the paths of $Q$ as morphisms).
The category of representations
$\Rep(Q,\bF_1)=[\cP(Q),\vect_{\bF_1}]$
(\cf \cite{szczesny_representations})
is parabelian by Lemma \ref{Fun}.
Let $A=\bar\cP(Q)$ be the pointed path semigroup of $Q$, with the product $pq$ of two paths defined by the concatenation of $p$ and $q$ if they are compatible and defined to be zero otherwise.
This semigroup doesn't have the identity unless $Q$ has only one vertex.
Moreover, unlike in the additive case, one can not identify $\Rep(Q,\bF_1)$ with the category of $A$-modules in $\vect_{\bF_1}$ unless $Q$ has only one vertex.
More precisely, there is a pair of adjoint functors
\[L:\Rep(Q,\bF_1)\to\Mod A,\ X\mto\bop_{i\in Q_0}X_i,\qquad
R:\Mod A\to\Rep(Q,\bF_1),\ M\mto (e_iM)_{i\in Q_0},
\]
where $e_i\in A$ is the trivial path at $i\in Q_0$.
There is an isomorphism $\id \iso RL$,
but the canonical map $LRM=\bop_{i\in Q_0}e_iM\to M$ is not surjective in general.

\subsubsection{Pre-crystals}
Using Kashiwara crystals \cite[\S7.2]{kashiwara_crystal} as a motivation, we define a \idef{pre-crystal}
to be a set $B$ equipped with the maps (for a fixed index set $I$)
\[e_i:B\to B_*=B\sqcup\set0,\qquad f_i:B\to B_*,\qquad i\in I,\]
such that $b'=f_ib$ \iff $b=e_ib'$ for $b,b'\in B$.
We can consider $e_i,f_i$ as morphisms $e_i:B_*\to B_*$, $f_i:B_*\to B_*$ of pointed sets.
Then the above axiom means that the morphisms $e_i,f_i$ are strict and dual to each other \eqref{dual}.
Therefore, we can define a pre-crystal to be a pointed set $X$ equipped with strict morphisms $e_i:X\to X$ for $i\in I$.
The original pre-crystal is obtained by considering $B=X\ms\set0$ and $f_i=e_i^*$.
In what follows we will denote $e_ax$ by $ax$
for $a\in I$, $x\in X$.

A morphism $f:X\to Y$ between pre-crystals is a pointed map such that if $f(x)\ne0$ and $f(ax)\ne0$, then $f(ax)=af(x)$ for $a\in I$
(\cf \cite[\S7.2]{kashiwara_crystal}).
Let $\PC$ denote the category of pre-crystals.
We will show that it has the same kernels and cokernels as $\Set_*$.

\begin{lemma}
The category $\PC$ has kernels and cokernels and the forgetful functor $F:\PC\to\Set_*$ is exact and reflects isomorphisms.
Therefore $\PC$ is \qpa.
\end{lemma}
\begin{proof}
It is clear that $F$ reflects isomorphisms.
For a morphism $f:X\to Y$ in $\PC$, we equip $K=f\inv(0)$ with the crystal structure by $a.x=ax$ if $ax\in K$ and zero otherwise.
The inclusion $K\emb X$ is a morphism of crystals: if $a.x\ne0$, then $a.x=ax$.
Given a morphism $U\xto g X$ with $fg=0$, we obtain the map $g:U\to K$ between pointed sets.
If $g(x),g(ax)\ne0$ for some $x\in U$, then $ag(x)=g(ax)\in K$, hence $a.g(x)=ag(x)=g(ax)$. Therefore $g$ is a morphism of crystals.
We conclude that $K=\Ker(f)$.

For a morphism $f:X\to Y$ in $\PC$, we equip $C=Y/f(X)$ with the crystal structure by $a.y=ay$ if $ay\in Y\ms f(X)$ and zero otherwise.
The projection $\pi:Y\to C$ is a morphism of crystals.
Indeed, if $\pi(y),\pi(ay)\ne0$, then $y,ay\in Y\ms f(X)$, hence $a.\pi(y)=ay=\pi(ay)$.
Given a morphism $g:Y\to U$ such that $gf=0$, we obtain the induced map $g':C\to U$ between pointed sets.
If $g'(y)\ne0$ and $g'(a.y)\ne0$, then $g(y)\ne0$ and $g(ay)\ne0$,
hence $g'(a.y)=g(ay)=ag(y)$.
Therefore $g'$ is a morphism of crystals.
We conclude that $C=\Coker(f)$.
\end{proof}

More generally, given a quiver $Q$, we define a $Q$-pre-crystal to be an object in $[\cP(Q),\Set_*^\s]$.
Explicitly, it is a collection of pointed sets $(X_i)_{i\in Q_0}$ and strict morphisms $X_a:X_i\to X_j$ for arrows $a:i\to j$ in $Q$ (we denote $X_a x$ by $ax$ for $a:i\to j$ and $x\in X_i$).
For example, brane tiling crystals \cite{mozgovoy_noncommutative}
provide instances of $Q$-pre-crystals.
Morphisms between $Q$-pre-crystals differ from morphisms in $[\cP(Q),\Set_*^\s]$.
Namely, a morphism $f:X\to Y$ of $Q$-pre-crystals is a collection of pointed maps $f_i:X_i\to Y_i$ for $i\in Q_0$ such that
if $f_i(x)\ne0$ and $f_j(ax)\ne0$ for $a:i\to j$ and $x\in X_i$,
then $f_j(ax)=af_i(x)$.
As before, one can show that the category of $Q$-pre-crystals is \qpa.

\subsection{Normed and Euclidean spaces}
\label{sec:normed}
Let $\NS$ be the category of normed finite-dimensional $\bR$-vector spaces
%(they are automatically Banach)
with linear maps of norm $\le1$ as morphisms.
This category has finite limits and colimits.
In particular,
\begin{enumerate}
\item The product $X\prod Y$ of $X,Y\in\NS$ is $X\oplus Y$ with the norm $\nn{(x,y)}=\max\set{\nn x,\nn y}$.
\item The coproduct $X\coprod Y$ of $X,Y\in\NS$ is $X\oplus Y$ with the norm $\nn{(x,y)}=\nn x+\nn y$.
\item The kernel of $f:X\to Y$ is $f\inv(0)$ with the induced norm.
\item The cokernel of $f:X\to Y$ is $Y/f(X)$ with the quotient norm
\[\nn{y+f(X)}=\inf\sets{\nn{y+z}}{z\in f(X)},\qquad y\in Y.\]
\end{enumerate}

\begin{theorem}
The category $\NS$ is \qpa.
\end{theorem}
\begin{proof}
\clm1 Given a cartesian square
\[\begin{tikzcd}
X'\rar["f'"]\dar["g'"']&Y'\dar["g"]\\
X\rar["f"]&Y
\end{tikzcd}\]
with strict mono $f$ and strict epi $g$,
we need to show that $g'$ is strict epi.
We know that $f'$ is strict mono.
We can identify $X$ with $f(X)\sbs Y$ and $X'$ with $g\inv(X)\sbs Y'$ equipped with the induced norm.
For $y\in Y$ we have $\nn{y}=\inf\sets{\nn{y'}}{g(y')=y}$.
Therefore, for $x\in X$,
we have $\nn{x}=\inf\sets{\nn{x'}}{g'(x')=x}$.
This implies that $g'$ is strict epi.

\clm2
Assuming that the above square is cocartesian, with strict mono $f'$ and strict epi $g'$,
we need to show that $f$ is strict mono.
The map $f$ is injective and the diagram is a bicartesian square of vector spaces.
Therefore we can identify $X$ with $f(X)$ and $X'$ with $g\inv(X)$.
Since $g'$ and $g$ are strict epi, we have, for $x\in X$,
$\nn x_X=\inf\sets{\nn{x'}}{g'(x')=x}
=\inf\sets{\nn{y'}}{g(y')=x}=\nn x_Y$.
\end{proof}

Let $\ES\sbs\NS$ be the full subcategory of finite-dimensional $\bR$-vector spaces equipped with an inner product.
The objects of $\ES$ are called Euclidean spaces and the induced norm is called the Euclidean norm.
The category $\ES$ doesn't have products or coproducts (the orthogonal direct sum $X\oplus Y$ equipped with the Euclidean norm
$\nn{(x,y)}=(\nn x^2+\nn y^2)^{1/2}$ is neither a product nor a coproduct in $\ES$).
But it has kernels and cokernels.

\begin{lemma}
The full subcategory $\ES\sbs \NS$ is closed under taking kernels and cokernels.
Therefore $\ES$ is \qpa.
\end{lemma}
\begin{proof}
The kernel and cokernel of $f:X\to Y$ in $\ES$ are
\begin{enumerate}
\item $f\inv(0)$ equipped with the induced inner product.
\item The orthogonal complement $f(X)^\perp\iso Y/f(X)$
equipped with the induced inner product.
\end{enumerate}
The induced norms of the above kernel and cokernel coincide with the norms of the kernel and cokernel in $\NS$.
\end{proof}

\begin{remark}
The above description of kernels and cokernels in $\ES$ implies that
a morphism $f:X\to Y$ is strict \iff $f:\Ker(f)^\perp\to f(X)$ is an isometry.
The class of all strict morphisms is not closed under composition.
The strict mono/epi in $\ES$ are exactly the split mono/epi.
The pullbacks of strict epimorphisms don't exist in general in $\ES$.
For example, the pullback of $X\to 0$ along $Y\to 0$ is the product $X\xx Y$ which doesn't exist in general in $\ES$.
\end{remark}

\begin{remark}
We can similarly define the category $\NS_\bC$ of normed finite-dimensional $\bC$-vector spaces and the full subcategory $\HS\sbs\NS_\bC$ of Hermitian spaces.
Both of them are \qpa.
\end{remark}

\subsection{Hermitian vector bundles}
If $\cO$ is a Dedekind domain, then the category $\proj(\cO)$ of finitely-generated projective $\cO$-modules is quasi-abelian, hence \qpa.
The strict mono/epi are exactly the split mono/epi (similarly to the categories $\ES$ and $\HS$).
Let $K$ be a number field, $\cO=\cO_K$ be its ring of integers and $\Si$ be the set of embeddings of $K$ in $\bC$.
A Hermitian vector bundle $(E,h)$ on $S=\Spec\cO$ consists of $E\in\proj(\cO)$ and a collection
$h=(h_\si)_{\si\in\Si}$ of Hermitian forms
on $E\ts_{\si}\bC$
such that $h_{\bar\si}(x,y)=h_\si(y,x)$
for $x,y\in E$ and $\si\in\Si$
(invariance under complex conjugation).
Let $\Bun(\bar S)$ be the category of Hermitian vector bundles, where morphisms $f:(E,h)\to(E',h')$ are $\cO$-linear maps $f:E\to E'$ such that $\nn{fx}_\si\le\nn x_\si$ for all $\si\in\Si$ and $x\in E$.
We will show that $\Bun(\bar S)$ is \qpa.

\begin{example}
For $K=\bQ$ and $\cO=\bZ$, there is only one embedding $\si:\bQ\to\bC$ and $\bar\si=\si$.
Therefore, for every $(E,h)\in\Bun(\ub{\Spec\bZ})$, the Hermitian form $h_\si$ corresponds to a (Euclidean) inner product on $E_\bR=E\ts_\bZ\bR$.
We can alternatively describe the category $\Bun(\ub{\Spec\bZ})$ as follows.
Consider the diagram $\Psi$ consisting of the exact functor
$\Psi_\si:\proj(\bZ)\to\Vect_\bR$, $E\mto E_\bR$,
and the exact forgetful functor $\Psi_{\si'}:\ES\to\Vect_\bR$.
Then $\Bun(\ub{\Spec\bZ})\iso\Rep^\bi(\Psi)$
and $\Rep^\bi(\Psi)$ is \qpa by Theorem \ref{Q-diagram}.
\end{example}

\begin{theorem}
Let $K$ be a number field and $\cO$ be its ring of integers.
Then the category $\Bun(\ub{\Spec\cO})$ of Hermitian vector bundles on $\Spec\cO$ is \qpa.
\end{theorem}
\begin{proof}
Consider the diagram $\Psi$ of exact functors
\[\Psi_\si:\cA_0=\proj(\cO)\to \cA_\si=\Vect_\bC,\qquad
E\mto E\ts_\si\bC,\]
and forgetful functors $\Psi_{\si'}:\cA_{\si'}=\HS\to\cA_\si=\Vect_\bC$ for $\si\in\Si$.
Then there is a fully faithful exact functor
$F:\Bun(\ub{\Spec\cO})\to \Rep^\bi(\Psi)$ that maps $(E,h)$ to $\bar E\in\Rep^\bi(\Psi)$ with $\bar E_0=E\in\cA_0$,
$\bar E_\si=E\ts_\si\bC\in\cA_\si$ and $\bar E_{\si'}=(\bar E_\si,h_\si)\in\cA_{\si'}$.
The category $\Rep^\bi(\Psi)$ is \qpa by Theorem \ref{Q-diagram}.
Therefore the category $\Bun(\ub{\Spec\cO})$ is \qpa by Lemma~\ref{subcat}.
\end{proof}

\subsection{Closure spaces and matroids}
\label{sec:closure}
%This section is motivated by \cite{eppolito_proto}, where an explicit proto-exact structure on the category of pointed matroids is constructed.
A \idef{closure space} is a set $X$ equipped with a
\idef{closure operator} $\si=\si_X:2^X\to 2^X$ satisfying
$\si^2=\si$ and $A\sbs\si(A)\sbs\si(B)$ for $A\sbs B\sbs X$.
A subset $F\sbs X$ is called closed if $\si F=F$.
The family $\cF\sbs 2^X$ of all closed sets satisfies
$\cap\cI\in\cF$ for all $\cI\sbs\cF$.
% (for a finite set $X$ this means that $\cF\sbs 2^X$ is a sublattice with the top element $X$).
Conversely, a family $\cF\sbs 2^X$ with this property defines the closure operator $\si(A)=\cap\sets{F\in\cF}{F\sps A}$ for $A\sbs X$.

\begin{lemma}
For a map $f:X\to Y$ between closure spaces the following are equivalent
\begin{enumerate}
\item $f\inv F$ is closed for all closed $F\sbs Y$.
\item $\si f\inv B\sbs f\inv\si B$ for all $B\sbs Y$.
\item $f\si A\sbs\si f A$ for all $A\sbs X$.
\end{enumerate}
A map satisfying these conditions is called continuous.
\end{lemma}
\begin{proof}
\dir12 For $F=\si B$, we have $\si f\inv(B)\sbs\si f\inv F=f\inv F=f\inv\si B$.
\dir23 For $B=f A$, we have $\si A\sbs \si f\inv B\sbs f\inv\si B=f\inv \si fA$, hence $f\si A\sbs\si f A$.
\dir31
For $A=f\inv F$, we have $\si A\sbs f\inv\si f A\sbs f\inv\si F=f\inv F=A$, hence $A$ is closed.
\end{proof}

\begin{remark}
Similarly, a map $f:X\to Y$ between closure spaces is closed ($fF$ is closed for all closed $F\sbs X$) \iff
$\si fA\sbs f\si A$ for all $A\sbs X$.
Indeed, if $f$ is closed, then for $A\sbs X$ and $F=\si A$ we have $\si f A\sbs \si fF=fF=f\si A$.
Conversely, if $\si fA\sbs f\si A$ for all $A\sbs X$,
then for a closed $F\sbs X$ we have $\si fF\sbs f\si F=fF$, hence $fF$ is closed.
A closed continuous map satisfies $f\si=\si f$.
\end{remark}

\begin{example}
For a topological space $X$, let $\si A\sbs X$ be the closure of $A\sbs X$.
Then $\si:2^X\to 2^X$ is a closure operator.
It defines the topology on $X$ uniquely.
\end{example}

\begin{example}
Let $(X,\cI)$ be a matroid, where $X$ is a finite set and $\cI\sbs 2^X$ is a collection of independent sets \cite{welsh_matroid}.
One defines the rank function
\[\rho:2^X\to\bZ,\qquad A\mto\max\sets{\n I}{I\sbs A,\, I\in\cI}\]
and the closure operator
\[\si:2^X\to 2^X,\qquad A\mto \sets{x\in X}{\rho(A\cup x)=\rho(A)}.\]
A subset $F\sbs X$ is closed (also called a flat)
\iff $\rho(F\cup x)>\rho(F)$ for all $x\in X\ms F$.
The above map $\si$ has the properties
\begin{enumerate}
\item $(X,\si)$ is a closure space.
\item If $y\in\si(A\cup x)\ms \si A$, then $x\in\si(A\cup y)$ for $A\sbs X$ and $x\in X$.
\end{enumerate}
Conversely, a map $\si$ with these properties defines a unique matroid structure on $X$.
A continuous map between matroids (considered as closure spaces) is called a strong map \cite[\S17]{welsh_matroid}.
\end{example}

Let $\Cl$ denote the category of closure spaces and continuous maps between them.
The forgetful functor $U:\Cl\to\Set$ has a fully faithful right adjoint functor $R:\Set\to\Cl$, where $RX=X$ has the closure operator $\si A=X$ for all $A\sbs X$.
The category $\Cl$ has the terminal object $R\set0=\set0$ with $\si\es=\set0$.
Let $\Cl_*=\set0/\!\Cl$ be the comma category, with morphisms $\set0\to X$ in $\Cl$ as the objects.
These objects can be identified with pointed closure sets $(X,0_X)$, where $X\in\Cl$ and $0_X\in\si\es\sbs X$ (the base point).
The pointed closure set $(\set0,0)$ is the zero object of $\Cl_*$, denoted by $0$.
Let $\Mat\sbs\Cl$ denote the full subcategory of matroids and $\Mat_*\sbs\Cl_*$ denote the corresponding pointed category.
The functor $R:\Set\to\Cl$ factors through $\Mat$.

\begin{lemma}
For every map $f:X\to Y$ in $\Cl_*$ there exists
\begin{enumerate}
\item
the kernel $i:K=f\inv(0_Y)\emb X$,
where $K$ has the closure operator
$\si_K=i\inv \si_X i$,
meaning that $\si_K(A)=K\cap \si_X(A)$ for $A\sbs K$.
\item the cokernel $\pi:Y\to C=Y/f(X)$, where $C$ has the closure operator $\si_C=\pi\si_Y\pi\inv$,
meaning that $F\sbs C$ is closed \iff $\pi\inv F$ is closed.
\end{enumerate}
\end{lemma}

\begin{proof}
We will show that $i$ is the kernel.
If $a:U\to X$ satisfies $fa=0$, then there is a unique morphism of pointed sets $b:U\to K$ with $ib=a$.
For $A\sbs U$, we have $ib\si_U(A)\sbs \si_X ib(A)$,
hence $b\si_U(A)\sbs\si_K b(A)$ and $b$ is continuous.
\end{proof}

We conclude by Lemma \ref{pull/push} that strict monomorphisms are closed under pullbacks along $\Cl_*$ and strict epimorphisms are closed under pushouts along $\Cl_*$.
The above lemma implies that the forgetful functor $U:\Cl_*\to\Set_*$ preserves kernels and cokernels, hence also preserves
pullbacks of strict monomorphisms and pushouts of strict epimorphisms.

\begin{theorem}
The category $\Cl_*$ is \qpa.
\end{theorem}
\begin{proof}
Given a cartesian square
\[\begin{tikzcd}
X'\rar["f'"]\dar["g'"']&Y'\dar["g"]\\
X\rar["f"]&Y
\end{tikzcd}\]
with strict mono $f$ and strict epi $g$,
we need to show that $g'$ is strict epi.
The map $g'$ is a strict epimorphism of pointed sets,
hence we only need to show that $\si_X=g'\si_{X'}(g')\inv$.
We can identify $X$ with $f(X)$ and $X'$ with $g\inv(X)$.
For $B\sbs X'$, we have $\si_{X'}B=X'\cap\si_{Y'}B$ (since $f'$ is a strict mono).
For $A\sbs X$ and $B=(g')\inv A=g\inv A$,
we have $g'\si_{X'}B=g'(X'\cap\si_{Y'}B)=g(g\inv X\cap\si_{Y'}B)=X\cap g\si_{Y'}B=X\cap g\si_{Y'}g\inv A
=X\cap\si_Y A=\si_XA$.

The proof for cocartesian diagrams is similar.
\end{proof}

\begin{lemma}
The full subcategory $\Mat_*\sbs\Cl_*$ is closed under taking kernels and cokernels.
Therefore the category $\Mat_*$ is \qpa.
\end{lemma}
\begin{proof}
Let $f:X\to Y$ be a morphism in $\Mat_*$.
Let $K=f\inv(0_Y)$ be equipped with the closure operator $\si_KA=K\cap\si_XA$ for $A\sbs K$.
To prove that $(K,\si_K)$ is a matroid, we need to show
that $y\in\si_K(A\cup x)\ms \si_K A$ implies $x\in\si_K(A\cup y)$ for $A\sbs K$ and $x,y\in K$.
This follows from the same property for $\si_X$.

Let $C=Y/X'$, for $X'=f(X)$, be equipped with the closure operator $\si_C=\pi\si_Y\pi\inv$, where $\pi:Y\to C$ is the projection.
If $\pi y\in\si_C(A\cup \pi x)\ms\si_C A$ for $A\sbs C$ and $x\in Y$, then $\pi y\ne0$ (since $0\in\si_CA$).
If $\pi x=0$, then $\pi x\in\si_C(A\cup\pi y)$, hence we can assume that $\pi x\ne0$.
For $B=\pi\inv A$, we have $\pi y\in\pi\si_Y(B\cup x)$, hence $y\in\si_Y(B\cup x)$.
We have $\pi y\notin\si_CA=\pi\si_Y B$, hence $y\notin\si_YB$.
As $(Y,\si_Y)$ is a matroid, we conclude that $x\in\si_Y(B\cup y)
=\si_Y\pi\inv(A\cup \pi y)$, hence $\pi x\in\si_C(A\cup\pi y)$.
This implies that $(C,\si_C)$ is a matroid.
\end{proof}

\begin{remark}
The above results imply that the categories $\Cl_*$ and $\Mat_*$ are equipped with the canonical \pe structure
consisting of all short exact sequences.
An explicit proto-exact structure on the category $\Mat_*$ was constructed in \cite{eppolito_proto}.
It coincides with the canonical one.
\end{remark}

\section{Finitary algebraic categories}
%\href{https://en.wikipedia.org/wiki/Variety_(universal_algebra)}{wiki}

\subsection{Finitary monads}
\label{fini}
Varieties of algebraic structures in universal algebra can be encoded using Lawvere theories~\cite{lawvere_functorial} or using \idef{finitary monads} on $\Set$ (monads that preserve filtered colimits)
\cite{linton_some,hyland_category}.
For a finitary monad $T$ on $\Set$, the Eilenberg-Moore category $\Set^T$ of $T$-algebras is called a \idef{finitary algebraic category}.
Our interest in such categories is motivated by \dur, where finitary monads on $\Set$ are called \idef{algebraic monads} and algebras over them are interpreted as modules over \qq{generalized rings}.
Many categories in Arakelov geometry arise in this way
\cite{durov_new,fresan_compactification}.
For example, the categories of pointed sets and normed spaces (with morphisms of norm $\le1$)
are finitary algebraic categories or their subcategories.
In this section we will study the question of when such categories are
\qpa.

Recall that a functor $U:\cA\to\cB$ is called \idef{monadic} if it has a left adjoint $F:\cB\to\cA$ such that for the monad $T=UF:\cB\to\cB$,
the canonical functor $\bar U:\cA\to\cB^T$, $X\mto UX$, is an equivalence.
By Beck's theorem (see \eg \cite{maclane_categories}),
%\cite[Ex.~VI.7.3]{maclane_categories}), [\S VI.7]
if $\cA$ has all coequalizers
and $U$ has a left adjoint,
preserves coequalizers
and reflects isomorphisms, then $U$ is monadic.
If $\cB$ is a category with all limits and colimits and $T$ is a monad on $\cB$ that preserves filtered colimits, then $\cB^T$ has all limits and colimits \cite{barr_toposes}.
For example, let $A$ be a ring (or a semiring \S\ref{sec:cmon})
and $\Mod(A)$ be the category of left modules over~$A$.
%an associative unital ring~$A$.
The forgetful functor $U:\Mod(A)\to\Set$ has the left adjoint
\begin{equation}\label{ring left adj}
F:\Set\to\Mod(A),\qquad X\mto AX=\coprod\nolimits_{x\in X}A.
\end{equation}
The monad
\begin{equation}\label{ring monad}
\bar A=UF:\Set\to\Set
\end{equation}
will be sometimes denoted by $A$.
The canonical functor $\bar U:\Mod(A)\to\Set^{\AA}$ is an equivalence.

Let $T$ be a finitary monad on the category $\Set$
(throughout, all finitary monads will be considered on \Set).
Let $\NN\sbs\Set$ be the full subcategory with objects $\bn=\set{1,\dots,n}$ for $n\ge0$ (equivalent to the category of finite sets).
We denote $T(\bn)$ by $T(n)$.
The functor $T:\Set\to\Set$ is uniquely determined by its restriction to $\NN$ \dur.
Namely, $TX=\dlim_{[\bn\to X]\in\NN/X}T(\bn)$ for $X\in\Set$.
Given a morphism $\al:TX\to Y$, for every $x\in X^n$ (corresponding to $x:\bn\to X$), we have a map $T(\bn)\xto{Tx} T(X)\xto\al Y$, hence
%\dur[\S4.2]
\[\al_n:T(n)\xx X^n\to Y,\qquad n\ge0.\]
For $t\in T(n)$ and $x\in X^n$, we denote $\al_n(t,x)$ by $t(x)$.
These maps should satisfy
%\dur[\S4.2.2]
\[[T(\vi)(t)](x)=t(\vi^* x),\qquad \vi:\bm\to\bn,\quad
t\in T(m),\quad x\in X^n,\quad \vi^*x=(x_{\vi i})_{i\in\bm}.\]

In particular, the product map $\mu:T^2\to T$ of the monad induces
\[\mu_m=\mu_m^{(n)}:T(m)\xx T(n)^m\to T(n),\qquad m,n\ge0.\]
We also have the unit map $\eta:\bn\to T(\bn)$
and define $e_i=e_i^{(n)}=\eta(i)\in T(n)$ for $i\in\bn$.
Note that for $T=\bar A$ from \eqref{ring monad}, the elements $e_i\in A^n$ are the standard basis elements.
Let $1=e_1\in T(1)$.
The axioms of a monad are \dur[\S4.3.3]
\begin{enumerate}
\item $1(t)=t$ for $t\in T(n)$.
\item $t(e_1,\dots,e_n)=t$ for $t\in T(n)$.
\item $s(t(x))=(s(t))(x)$ for $s\in T(m)$, $t\in T(n)^m$ and $x\in T(k)^n$, where $t(x)=(t_i(x))_{i\in\bm}$.
\end{enumerate}

For a $T$-algebra $X$ with the structure map $\al:TX\to X$,
we have the maps $\al_n:T(n)\xx X^n\to X$ and define $t(x)=\al_n(t,x)$ for $t\in T(n)$, $x\in X^n$ as before.
The axioms of an algebra are
%\dur[4.3.2]
\begin{enumerate}
\item $1(x)=x$ for $x\in X$.
This also implies $e_i(x)=x_i$ for $x\in X^n$ and $i\in\bn$.
\item $s(t(x))=(s(t))(x)$ for $s\in T(m)$, $t\in T(n)^m$ and
$x\in X^n$, where $t(x)=(t_i(x))_{i\in\bm}$.
\end{enumerate}

\begin{lemma}\label{lm:submon}
Let $X$ be an algebra over a finitary monad $T$ and $S\sbs X$ be a subset.
Then there is a finitary submonad $T_S\sbs T$ (stabilizer of $S$) defined by
\[T_S(n)=\sets{t\in T(n)}{t(S^n)\sbs S},\qquad n\ge0.\]
\end{lemma}
\begin{proof}
For $s\in T_S(m)$ and $t\in T_S(n)^m$, we need to show that $u=s(t)\in T(n)$ is contained in $T_S(n)$.
For $x\in S^n$, we have $u(x)=s(t(x))$.
As $t\in T_S(n)^m$, we obtain $t(x)\in S^m$,
hence $s(t(x))\in S$.
This implies that $u(S^n)\sbs S$, hence $u\in T_S(n)$.
\end{proof}

\begin{remark}
The set $\n T=T(1)$ has the monoid structure $s\cdot t=s(t)$ for $s,t\in T(1)$, with the identity $1\in T(1)$.
It is called the \idef{underlying monoid} of $T$.
If $S\sbs T(1)$ is a submonoid,
% (note that $T(1)$ is a $T$-algebra),
then $T_S(1)=S$, hence the underlying monoid of $T_S$ is $S$.
\end{remark}

Since $T$ preserves filtered colimits, the finitary algebraic category $\Set^T$ has all limits and colimits.
The limits are constructed  in the category of sets.
Let us also describe the construction of coequalizers \dur.
For $X\in\Set^T$, an equivalence relation $R\sbs X\xx X$ is said to
be $T$-compatible
%be compatible with the $T$-structure
if $X/R$ admits (a unique) $T$-algebra structure such that the projection map $\pi:X\to X/R$ is a morphism of $T$-algebras.
This means that if $x,y\in X^n$ satisfy $x_i\sim_R y_i$ for all $i\in\bn$, then
$t(x)\sim_R t(y)$ for all $t\in T(n)$.
The intersection of $T$-compatible equivalence relations on $X$ is $T$-compatible \dur[\S4.6], hence for every subset $R\sbs X\xx X$ there exists the  minimal $T$-compatible equivalence relation
$\ang R=\ang R_X\sbs X\xx X$ that contains $R$, and we have the $T$-algebra $X/\ang R$.
In particular, the coequalizer of two morphisms $f,g:X\to Y$ is defined to be $Y/\ang R$ for $R=\sets{(f(x),g(x))}{x\in X}\sbs Y\xx Y$.

Given a homomorphism $\rho:S\to T$ of algebraic monads,
there is the scalar restriction functor $\rho^*:\Set^{T}\to\Set^S$,
which admits the left adjoint $\rho_*:\Set^S\to\Set^T$, $X\mto X_T=T\ts_{S}X$, called the scalar extension functor \dur[\S4.6.19].
%There is a canonical morphism $X\to X_R$ of $S$-modules.
%\dur[\S4.6.19].

The category $\Set^T$ has the terminal object $\set0$ and the initial object $T(0)=T(\es)$.
We say that~$T$ is \idef{pointed} (or a \idef{monad with zero} \dur) if $T(0)=\set0$.
%[\S 4.3.12]
In this case, the category $\Set^T$ has the zero object $T(0)=\set0$ and admits kernels and cokernels.
For a $T$-algebra $X$, the structure map $\al_0:T(0)=\set0\to X$
makes $X$ a pointed set with the base point $0_X=\al_0(0)\in X$.
Given a $T$-subalgebra $X\sbs Y$, we define the quotient algebra
$Y/X=Y/\ang{X\xx X}$, where $\ang{X\xx X}=\ang{X\xx 0_X}\sbs Y\xx Y$ is the $T$-compatible equivalence relation on $Y$ generated by $X\xx X$.
Define the \idef{saturation} of $X$ to be $\ang X=\pi\inv(0)$ for the projection map $\pi:Y\to Y/X$.
It is a $T$-algebra containing~$X$.
A subalgebra $X\sbs Y$ is called \idef{saturated} if $X=\ang X$,
meaning that $\ang{X\xx X}\cap(X\xx Y)=X\xx X$.
For a morphism $f:X\to Y$ of $T$-algebras, we have
\begin{equation}\label{ker-coker-T}
\begin{gathered}
\Ker(f)=f\inv(0_Y),\quad
\Coker(f)=Y/f(X),\\
\Coim(f)=X/f\inv(0_Y),\quad \Im(f)=\ang{f(X)}.
\end{gathered}
\end{equation}
The maps $\Coim(f)\onto f(X)\emb\Im(f)$ are not necessarily isomorphisms (see Example \ref{ex:cmon Im/Coim}).
Let us say that a morphism of $T$-algebras is injective/surjective if this is the case for the underlying map between sets.
Then
\[\text{strict mono/epi} \imp\text{injective/surjective} \imp\text{mono/epi}.\]
A morphism $f:X\to Y$ is strict mono \iff $f$ is injective and $f(X)$ is saturated in $Y$.

\begin{example}
The monad $\bF_1$ from \S\ref{sec-pointed} is finitary and the category $\Set^{\bF_1}\iso\Set_*$ is \qpa.
\end{example}

\begin{example}\label{ex-cmon}
Let $\CMon$ be the category of commutative monoids (modules over the semiring~$\bN$).
The forgetful functor $U:\CMon\to\Set$ is monadic, with the left adjoint $F:\Set\to\CMon$, $X\mto\bN X=\coprod_{x\in X}\bN$.
The monad $\bN=UF:\Set\to\Set$ is finitary and pointed, and $\CMon\iso\Set^\bN$.
We will see later that the category $\CMon$ is \qpa.
\end{example}

\begin{example}
Let $\Grp$ be the category of groups.
The forgetful functor $U:\Grp\to\Set$ is monadic, with the left adjoint $F:\Set\to\Grp$ being the free group functor.
The monad $T=UF$ is finitary and pointed,
and the category $\Grp\iso\Set^T$ is pointed
and has all limits and colimits.
A morphism $f:X\to Y$ is a strict mono \iff it is an embedding of a normal subgroup.
The composition of strict mono is not necessarily strict mono.
Therefore the category $\Grp$ is not \qpa.
\end{example}

\begin{example}
\label{Zinf-1}
Let $\Norm_1$ be the category of all normed $\bR$-vector spaces with linear maps of norm $\le1$ as morphisms (\cf \S\ref{sec:normed}).
The functor
\[U:\Norm_1\to\Set,\qquad E\mto E_{\le1}=\sets{x\in E}{\nn x\le1},\]
has the left adjoint $F:\Set\to\Norm_1$, $X\mto(\bR X,\nn\blank_1)$
and we obtain the finitary monad
\[\ZZ=UF:\Set\to\Set,\qquad X\mto \sets{t\in\bR X}{\nn t_1\le1}.\]
The canonical functor $\bar U:\Norm_1\to\Mod(\ZZ)=\Set^{\ZZ}$ is fully faithful, but not an equivalence \cite{semadeni_monads,pumpluen_banach,boerger_cogenerators,durov_new}.
The algebras over the monad $\ZZ$ are called
finitely totally convex spaces~\cite{pumpluen_banach},
absolutely convex spaces~\cite{boerger_cogenerators}
or $\ZZ$-modules \dur (see also Example \ref{Zinf-2}).
We have seen that $\NS\sbs\Mod(\ZZ)$ is \qpa and we will prove later that so is $\Mod(\ZZ)$.

For $X\in\Mod(\ZZ)$, $t\in\ZZ(n)\sbs\bR^n$ and $x\in X^n$
one denotes $t(x)\in X$ by $\sum_i t_ix_i$.
The structure of a $\ZZ$-module $X$ is uniquely determined by the involution
$X\mto X$, $x\mto-x=(-1)x$,
and the convex combinations
$[0,1]\xx X^2\to X$, $(t,x,y)\mto (1-t)x+ty$.

Since $\ZZ\sbs\bR$,
there is the scalar extension functor $\Mod(\ZZ)\to\Mod(\bR)$, $X\mto X_\bR$.
For example, for a normed space $E$, the set $X=E_{\le1}$ is a $\ZZ$-module and $X\to X_\bR=E$ is injective.
But generally the canonical map $X\to X_\bR$ is not injective.
For example, consider $\ZZ$-modules $(-1,1)\sbs[-1,1]$ and their quotient $Q=\set{-1,0,1}$ equipped with convex combinations
$(1-t)x+ty=0$ for $t\in (0,1)$, $x$ for $t=0$ and $y$ for $t=1$.
Then $Q_\bR=0$.
The subcategory of objects $X\in\Mod(\ZZ)$ such that $X\to X_\bR$ is injective is called the category of flat $\ZZ$-modules and is equivalent to $\Ind(\NS)$~\dur[\S2.7].
The category $\NS\sbs\Norm_1\sbs\Ind(\NS)\sbs\Mod(\ZZ)$ is interpreted as the category of finitely-generated flat $\ZZ$-modules.
This should be compared with Lazard's theorem stating that a module over a commutative ring is flat if and only if it is a filtered colimit of finitely-generated flat (or free) modules.
\end{example}

\begin{remark}
For a monomorphism $f:X\emb Y$, let us denote $\Coker(f)$ by $Y/X$.
If $\cC$ is \qpa and $X\xto fY\xto gZ$ are strict mono,
then the induced morphism $\bar g:Y/X\to Z/X$ is strict mono.
Indeed, we have a diagram with cocartesian squares
\[\begin{tikzcd}
X\rar["f"]\dar&Y\rar["g"]\dar["h"]&Z\dar\\
0\rar&Y/X\rar["\bar g"]&Z/X
\end{tikzcd}\]
implying that $\bar g$ is the pushout of a strict mono $g$ along the strict epi $h=\coker(f)$.
Therefore $\bar g$ is strict mono.
In the case of $\cC=\Set^T$ this implies that $\bar g$ is injective.
\end{remark}

\begin{lemma}\label{para-crit}
Let $T$ be a finitary monad with zero such that for all strict mono
$X\emb Y\emb Z$ in $\Set^T$, the map $Y/X\to Z/X$ is injective.
Then $\Set^T$ is \qpa.
\end{lemma}
\begin{proof}
Given a cartesian square
\[\begin{tikzcd}
X'\rar["f'"]\dar["g'"']&Y'\dar["g"]\\
X\rar["f"]&Y
\end{tikzcd}\]
with strict mono $f$ and strict epi $g$,
we know that $f'$ is strict mono and we need to show that $g'$ is strict epi.
For $K=\Ker(g')\iso\Ker(g)$,
we have strict mono $K\emb X'\emb Y'$,
hence an injective map $X'/K\emb Y'/K\iso Y$.
It factorizes as $X'/K\to X\to Y$, hence $X'/K\to X$ is injective.
The map $g$ is surjective, hence $g'$ and $X'/K\to X$ are surjective.
Therefore $X'/K\to X$ is an isomorphism.

Assume that the square is cocartesian, $f'$ is strict mono and $g'$ is strict epi.
Then $g$ is strict epi and we need to show that $f$ is strict mono.
For $h=\coker(f)$, we have $hg=\coker(f')$ and $f'=\ker(hg)$.
%We have injective $X\to h\inv(0)$ and we claim that it is surjective.
The map $g$ is surjective, hence $X'\iso (hg)\inv(0)\to h\inv(0)$ is surjective. Therefore $X\to h\inv(0)$ is surjective.
For $\vi=\ker(g'):K\to X'$, we have $g'=\coker(\vi)$, hence
$g=\coker(f'\vi)$.
Therefore $f:X\iso X'/K\to Y\iso Y'/K$ is injective,
hence $f:X\to h\inv(0)$ is injective.
This implies that $f:X\to h\inv(0)$ is an isomorphism.
\end{proof}

\begin{remark}[Commutative monads]
\label{com-monads}
A finitary monad $T$ is called commutative if for all $s\in T(m)$, $t\in T(n)$ and $x=(x_{ij})\in X^{mn}$ (for all $X\in\Set^T$),
we have $s(x')=t(x'')$, where
%$x'\in X^m$ and $x''\in X^n$ are defined by
$x'_i=t(x_{i*})$ and $x''_j=s(x_{*j})$ for $i\in\bm$, $j\in\bn$.
It is enough to consider only $x_{ij}=e_{ij}\in X=T(mn)$ in the above definition.
It is proved in \dur[\S5.3] that under these conditions the category $\Set^T$ has a canonical structure of a closed symmetric monoidal category.
The tensor product $X\ts Y$ is a $T$-algebra such that $\Hom(X\ts Y,Z)$ parametrizes $T$-bilinear maps $\vi:X\xx Y\to Z$ (meaning that $\vi(x,\blank)$ and $\vi(\blank,y)$ are homomorphisms of $T$-algebras for all $x\in X$ and $y\in Y$).
The inner \Hom object is $\Hom_T(X,Y)\sbs Y^X$ equipped with the canonical $T$-algebra structure.

For example, for $T=\bF_1$ and $\Set^T=\Set_*$, the tensor product $X\ts Y$ of $X,Y\in\Set_*$ is the smash product $X\m Y$ and the inner \Hom is $\Hom(X,Y)$ equipped with the base point $0_{XY}$.
\end{remark}

\subsection{Commutative monoids and semirings}
\label{sec:cmon}

As in Example \ref{ex-cmon}, let $\CMon$ be the category of commutative monoids (written additively).
It corresponds to the finitary monad $\bN:\Set\to\Set$,
hence the category $\CMon\iso\Set^\bN$
is pointed and has all limits and colimits (see \S\ref{fini}).
The limits are constructed set-theoretically.

For a submonoid $X\sbs Y$, the equivalence relation $\sim_X=\ang{X\xx X}$ on $Y$ is given by (\cf~\cite{jun_proto})
\[y_1\sim_X y_2\qquad\iff\qquad y_1+x_1=y_2+x_2\text{ for some }x_1,x_2\in X.\]
The set $Y/X$ of equivalence classes inherits the monoid structure.
The saturation of $X$ is the equivalence class of $0$
\[\ang X=\sets{y\in Y}{y+x\in X\text{ for some }x\in X}\sps X.\]
Kernels and cokernels are constructed using \eqref{ker-coker-T}.

\begin{example}\label{ex:cmon Im/Coim}
For $f:X\to Y$, the canonical morphisms $\Coim(f)\to f(X)\to\Im(f)$ don't have to be isomorphisms.
For example, for $f:\bN^2\to\bN$, $(x,y)\mto x+y$,
we have $f\inv(0)=\set{0}$, hence $\Coim(f)=\bN^2\to f(\bN^2)=\bN$ is not an isomorphism.
For the embedding $f:X=\bN\ms1\emb\bN$,
we have $f(X)=X\ne\Im(f)=\ang{X}=\bN$.
\end{example}

\begin{theorem}
The category $\CMon$ is \qpa.
\end{theorem}

\begin{proof}
By Lemma \ref{para-crit} we need to show that given saturated submonoids $X\sbs Y$ and $Y\sbs Z$, the induced map $Y/X\to Z/X$ is injective.
This follows from the fact that the equivalence relation $\sim_X$ on $Y$ is induced by the equivalence relation $\sim_X$ on $Z$.
\end{proof}

The monad $\bN$ is commutative, hence the category $\CMon$ is equipped with the canonical structure of a closed symmetric monoidal category by Remark \ref{com-monads}.
A monoid in $(\CMon,\ts)$ is called a \idef{semiring}.
Explicitly, a semiring is a triple $(A,+,\cdot)$ such that
$(A,+)$ is a commutative monoid (with the zero element $0$), $(A,\cdot)$ is a monoid (with the identity $1\in A$)
and the operations satisfy distributivity and $a0=0a=a$ for all $a\in A$.
A semiring~$A$ defines the monad $T:\CMon\to\CMon$, $X\mto A\ts X$ and the category $\Mod(A)=\CMon^T$ of $A$-modules.
Explicitly, a module over $A$ is a commutative monoid $(M,+)$ equipped with a bilinear product $\cdot:A\xx M\to M$
that satisfies associativity and $1x=x$ for all $x\in M$.

\begin{corollary}[\cf {\cite{jun_proto}}]
Given a semiring $A$, the category $\Mod(A)$ is \qpa.
\end{corollary}
\begin{proof}
This follows from the previous result and
Corollary \ref{parab-modules1}.
\end{proof}

\begin{remark}[Additive monads]
Let $T$ be a finitary monad with zero
and let $0_n$ denote the base point of $T(n)$.
Define the comparison maps $\pi_n:T(n)\to T(1)^n$,
where $\pi_n(t)_i=t(0_1,\dots,1,\dots,0_1)$ for $t\in T(n)$ and $i\in \bn$, with $1\in T(1)$ placed in the $i$-th position.
The monad $T$ is called \idef{additive} (resp.\ \idef{hypoadditive}) if the maps $\pi_n$ are bijective (resp.\ injective) \dur[\S4.8].
It is proved in \dur that a monad $T$ is additive \iff the category $\Set^T$ is equivalent to the category of modules over a semiring.
This implies that the category $\Set^T$ is \qpa for an additive monad $T$.
It is unclear if the same is true for hypoadditive monads.
Below we will consider a class of such monads.
\end{remark}

\subsection{Special finitary monads}
The class of all finitary monads is too broad, while the class of additive monads is too restrictive, as they correspond exactly to semirings.
In this section we introduce an intermediate class of finitary monads
associated with multiplicative submonoids in semirings,
which encompasses most of the monads we are interested in.

Let $A$ be a semiring and $\AA$ be its finitary monad such that $\Mod(A)\iso\Set^{\AA}$ \eqref{ring monad}.
We have $\AA(n)=A^n$ and the products
\[\mu:\AA(m)\xx \AA(n)^m\to \AA(n),\qquad (s,t)\mto s(t)=
\rbr{\sum*_{i\in\bm} s_it_{ij}}_{j\in\bn}.\]
The unit $\eta:\bn\to \AA(n)$ is given by $i\mto e_i\in A^n$ (the standard basis vector).
By Lemma \ref{lm:submon}, for a multiplicative submonoid $S\sbs A=\AA(1)$,
there is a submonad $\bar S=\AA_S\sbs \AA$ defined by
\begin{equation}\label{submonad1}
\bar S(n)=\sets{t\in A^n}{t(x)=\sum*_{i}t_ix_i\in S\ \forall x\in S^n}.
\end{equation}
It satisfies $\bar S(1)=S$.
If $0\in S$, then $\bar S$ is a pointed monad and $\bar S(n)\sbs S^n$.
We define the category of $S$-modules $\Mod(S)=\Set^{\SS}$.
If $S\sbs A$ is a semiring, then $\bar S(n)=S^n$, hence $\bar S$ coincides with the monad \eqref{ring monad}
and $\Mod(S)$ coincides with the category of $S$-modules introduced in \S\ref{sec:cmon}.
By abuse of notation, we will sometimes denote the monad $\SS$ by $S$.

\begin{remark}[Weak partial semirings]
\label{WPS}
To make the construction of the monad  $\bar S$ depend
intrinsically on a multiplicative submonoid $0\in S\sbs A$,
we can equip $S$ with the partial addition
$a\bplus b=a+b$ if $a+b\in S$ and~$\infty$ otherwise (for a new absorbing element $\infty$).
The resulting structure $(S,\bplus,\cdot)$
is slightly weaker than the structure of a \idef{partial semiring} (\cf \cite{manes_inverse}),
defined to be a triple $(B,+,\cdot)$
such that $(B\sqcup\infty,+)$ is a commutative semigroup with the zero element $0\in B$ and the absorbing element $\infty$
and $(B,\cdot)$ is a monoid satisfying
\begin{enumerate}
\item $a0=0a=0$ for all $a\in B$.
\item
$a(b+c)=ab+ac$ and $(b+c)a=ba+ca$ if $b+c\ne\infty$ for $a,b,c\in B$.
\end{enumerate}
In our case, we have associativity $(a\bplus b)\bplus c=a\bplus (b\bplus c)$ only if both sides are different from~$\infty$.
For example, for $S=[-1,1]\sbs\bR$, we have $(-1\bplus 1)\bplus 1=1$, while $-1\bplus (1\bplus 1)=\infty$.
For $t\in S^n$, let $\sum_{i=1}^nt_i=t_1+\sum_{i=2}^nt_i$.
Then the monad $\bar S$ is defined by
\[\bar S(n)
=\sets{t\in S^n}{t(x)=\sum*_{i}t_ix_i\in S\ \forall x\in S^n}.
\]
Note that we obtain the same monad if $S\sbs A$ is equipped with a different partial addition $t_1\bplus t_2=t_1+t_2$ if $t(S^2)\sbs S$ and $\infty$ otherwise.
\end{remark}

For $X\in\Mod(S)$, we have structure maps
$\al:\SS(n)\xx X^n\to X$
and we denote $\al(t,x)$ by $t(x)=\sum_{i=1}^n t_ix_i$
for $t\in \bar S(n)\sbs S^n$ and $x\in X^n$.
For an $S$-submodule $X\sbs Y$, the equivalence relation $\sim_X=\ang{X\xx X}$ on $Y$
is generated by relations
\begin{equation}\label{rel WPS}
\sum*_i s_ix_i+\sum*_jt_jy_j\sim_X \sum*_jt_jy_j
\end{equation}
for $(s,t)\in \bar S(m+n)\sbs S^m\xx S^n$, $x\in X^m$ and $y\in Y^n$.

\begin{example}[Affine and convex spaces]
For a semiring $A$ and the monoid $D=\set1\sbs A$,
the monad $D_A=\bar A_D\sbs \bar A$ is given by $A$-valued distributions
\[D_A(n)=\sets{t\in A^n}{\sum*_i t_i=1}.\]
For $X\in\Mod(D_A)$, the structure maps $\al:D_A(n)\xx X^n\to X$
are uniquely determined by the map
$A\xx X^2\to X$, $(t,x,y)\mto (1-t)x+ty$.
If $A$ is a field, then $\Mod(D_A)$ is the category of affine spaces over $A$.
If $A=\bR_+=\bR_{\ge0}$ (with the usual addition and multiplication), then the monad $D=\bar A_D\sbs\bar A$ is given by
\[D(n)=\sets{t\in \bR_+^n}{\sum*_{i}t_i=1}\]
and is called the \idef{distribution monad}.
The modules over $D$ are called \idef{convex spaces}
\cite{stone_postulates,boerger_cogenerators,fritz_convex,jacobs_convexity}.
For example, if $X$ is a convex set in a real vector space, then $X$ has a structure of a convex space
given by $t(x)=\sum_{i=1}^n t_ix_i$ for $t\in D(n)$ and $x\in X^n$.
\end{example}

\begin{example}[Pointed convex spaces]
\label{p-conv}
For the pointed submonoid $D_*=[0,1]\sbs \bR_+$,
the corresponding monad $D_*\sbs\bar\bR_+$ is given by
\[D_*(n)=\sets{t\in\bR^n_+}{\sum*_i t_i\le1}\iso D(n+1)\]
and is called the \idef{sub-distribution monad}.
The modules over $D_*$ can be identified with \idef{pointed convex spaces}  (convex spaces $X$ equipped with a base point $0_X$).
For $X\in\Mod(D_*)$, the base point is the image of $\al:D_*(0)\xx X^0=\set0\to X$.
Conversely, a pointed convex space $(X,0_X)$ is equipped with the $D_*$-module structure by the formula
$t(x)=t_0 0_X+\sum_{i=1}^n t_ix_i$,
for $t\in D_*(n)$, $x\in X^n$ and $t_0=1-\sum_{i=1}^n t_i$.
In \cite{boerger_cogenerators} the objects of $\Mod(D_*)$
are called (finitely) positively convex spaces.
For $X\in\Mod(D_*)$, the scalar extension
$X_{\bR_+}\in\Mod(\bR_+)$ (a convex cone)
is the quotient of $\bR_+\xx X$
by the relations $(t,x)\sim(ts\inv,sx)$ for $s\in(0,1]$ and $(t,0)\sim(0,x)$ for $t\in\bR_+,x\in X$.
The induced map $X\to X_{\bR_+}$ is not necessarily injective.
For example, consider $X=\set{0,1}$ with $(1-t)0+t1=\de_{t,1}$
(it is the quotient of the embedding $[0,1)\emb[0,1]$).
Then $X_{\bR_+}=0$.
\end{example}

\begin{remark}
The above relation between the monads $D$ and $D_*$ can be generalized as follows.
For a monad $T$ on $\Set$,
we can define the monad $T_*$ by the formula $T_*(X)=T(X_*)$, where $X_*=X\sqcup 0_X$.
\end{remark}

\begin{example}[Absolutely convex spaces]
\label{Zinf-2}
For the pointed submonoid $\ZZ=[-1,1]\sbs\bR$,
the corresponding monad $\ZZ\sbs\bar\bR$ is given by
(\cf Example \ref{Zinf-1})
\[\ZZ (n)=\sets{t\in\bR^n}{\nn t_1\le1}.\]
Indeed,
$\ZZ(n)=\sets{t\in\bR^n}{\n{t(x)}\le1\text{ if }\nn x_\infty\le1}=\sets{t\in\bR^n}{\nn t_1\le1}$ since $\nn-_1$ is dual to the norm $\nn-_\infty$.
The notation is justified by the fact that
$\ZZ=\sets{x\in\bQ_\infty}{\n x_\infty\le1}$,
where $\n-_\infty=\n-$ is the usual norm on $\bQ$ and $\bQ_\infty=\bR$ is the completion
\cite{durov_new,fresan_compactification}.
Similarly, for a prime number $p$ and the $p$-adic norm $\n-_p$ on $\bQ$, we have $\bZ_p=\sets{x\in\bQ_p}{\n x_p\le1}$.

A $\ZZ$-module $X$ gives rise to a pointed convex space equipped with an involution $X\to X$, $x\mto -x=(-1)x$,
such that $t(-x)=-t(x)$ and if $x=-x$, then $x=0_X$.
%$0_X=-0_X$
This is an object in $\Mod(D_*)$ equipped with an action of $\mu_2=\set{\pm1}$.
Therefore we have a fully faithful functor $F:\Mod(\ZZ)\to\Mod(D_*)^{\mu_2}$.
The image consists of $X\in\Mod(D_*)^{\mu_2}$ such that $\mu_2$ acts freely on $X\ms0_X$.
The functor $F$ is exact.
\end{example}

\begin{example}
Given a monoid $G$, let $\Set_*^G=[G,\Set_*]$ be the category of pointed sets equipped with a $G$-action (see Remark \ref{mod over monoids}).
The forgetful functor $U:\Set_*^G\to\Set$ has the left adjoint $F:\Set\to\Set_*^G$, $X\mto \bigsqcup_{x\in X}G\sqcup\set0$.
The corresponding monad $T=UF$ is given by
$T(n)=G^n\sqcup\set0$.
On the other hand, let us embed $G_*=G\sqcup\set0$ into the monoid ring $\bZ G$
(this corresponds to the partial semiring structure on $G_*$ with the zero element $0\in G_*$ and $g+h=\infty$ for $g,h\in G$).
The monad $\bar G_*$ is given by
(using $e_i(g)\in G_*^n$ with $e_i(g)_j=g$ for $j=i$ and zero otherwise)
\[\bar G_*(n)
=\sets{e_i(g)}{g\in G,\, i\in\bn}\sqcup\set0
\iso G^n\sqcup \set0=T(n).\]
We conclude that $\bar G_*\iso T$ and $\Mod(G_*)\iso\Set_*^G$.
This category is \qpa by Lemma \ref{Fun}.

In particular, for $G=\set1$ and the partial semiring $\bF_1=G_*=\set{0,1}$ (with the sum $1+1=\infty$),
we obtain $\Mod(\bF_1)\iso\Set_*$.
For the partial semiring $\bF_{1^r}=(\mu_r)_*$ (with the sum $g+h=\infty$ for $g,h\in\mu_r$),
we obtain $\Mod(\bF_{1^r})\iso\Set_*^{\mu_r}$.
\end{example}

\begin{example}\label{mu_n}
For the multiplicative submonoid $S=\mu_r\sqcup\set0\sbs\bC$,
% with the induced partial semiring structure.
the monad is given by
\[\bar S(n)=\sets{t\in S^n}{t(S^n)\sbs S}
=\sets{e_i(g)\in S^n}{g\in\mu_r,\, i\in\bn}\sqcup\set0\iso\bar \bF_{1^r}(n).\]
Therefore $\bar S\iso\bar\bF_{1^r}$ and $\Mod(S)\iso\Mod(\bF_{1^n})\iso\Set_*^{\mu_r}$.
Note that the addition operations on $S$ and $\bF_{1^r}$ are different for even $r$ (we have $-1+1=0$ in $S$ and $-1+1=\infty$ in $\bF_{1^r}$).
\end{example}

\begin{example}
The (multiplicative) monoid $S=\set{0,1}$ can have three different partial semiring structures depending on the value of $1+1$.
\begin{enumerate}
\item If $1+1=0$, then $S=\bF_2=\bZ/2\bZ$ is a field and $\Mod(S)=\Mod(\bF_2)$ is abelian.
\item If $1+1=1$, then $S$ is the boolean semiring $B=\set{0,1}$ with addition $\vee$ and multiplication~$\m$.
The category $\Mod(S)=\Mod(B)$ is \qpa.
\item If $1+1=\infty$, then $\bar S=\bar \bF_1$
and $\Mod(S)=\Mod(\bF_1)\iso \Set_*$ is \qpa.
\end{enumerate}
\end{example}

\subsection{Valuation monads}
\label{val-mon}
A \idef{seminorm} on a semiring $K$ is a map $\n\blank:K\to\bR_+$ such that
\begin{enumerate}
\item $\n 0=0$ and $\n 1=1$.
\item $\n{xy}\le \n x \n y$.
%for all $x,y\in K$.
\item $\n{x+y}\le\n x+\n y$.
%for all $x,y\in K$.
\end{enumerate}
It is called
\begin{enumerate}
\item \idef{multiplicative} if $\n{xy}=\n x\n y$.
%for all $x,y\in K$.
\item a \idef{norm} if $\n x=0$ \iff $x=0$.
%for all $x\in K$.
\item a \idef{valuation} if it is a multiplicative norm.
\item \idef{non-archimedean} if $\n{x+y}\le\max\set{\n x,\n y}$.
%for all $x,y\in K$.
\end{enumerate}

Given a semiring $K$ equipped with a seminorm,
we consider the multiplicative monoid
\[\cO_K=\sets{x\in K}{\n x\le1}\]
equipped with the partial addition as in Remark \ref{WPS}
and the corresponding (finitary and pointed) monad $\bar\cO_K\sbs \bar K$ (also denoted by $\cO_K$)
\begin{equation}
\bar\cO_K(n)=\sets{t\in K^n}{t(\cO_K^n)\sbs\cO_K},\qquad n\ge0,
\end{equation}
called the \idef{valuation monad} of $K$.
The category of $\cO_K$-modules is $\Mod(\cO_K)=\Set^{\bar\cO_K}$.
If the seminorm is non-archimedean,
then $\cO_K$ is a semiring and $\bar\cO_K(n)=\cO_K^n$,
hence $\Mod(\cO_K)$ is \qpa.
One may ask if $\Mod(\cO_K)$ is \qpa in general.
For example,
\begin{enumerate}
\item For $K=\bN$ with the usual norm, we have $\cO_K=\set{0,1}=\bF_1$
and $\Mod(\cO_K)\iso\Set_*$.

\item For a semiring $K$ with the trivial valuation $\n x=1$ for $x\ne0$, we have $\cO_K=K$.

\item For a monoid $G$ and the semiring $K=\bN G$ with the valuation $\n{\sum_g x_g g}=\sum_g x_g$,
we have $\cO_K=G\sqcup \set0=G_*$
and $\Mod(\cO_K)\iso\Set_*^G$.
% the category of pointed sets equipped with a $G$-action.

\item For $K=\bR_+$, we have $\cO_K=[0,1]=D_*$ and the modules are pointed convex spaces.

\item For $K=\bR$, we have $\cO_K=[-1,1]=\ZZ$
and the modules are absolutely convex spaces.
\end{enumerate}
We have seen earlier that the categories in the first three examples are \qpa and we will see below that the categories $\Mod(D_*)$ and
$\Mod(\ZZ)$ are also \qpa.

We can also define another monad $\cO_K'\sbs\bar K$ by
\begin{equation}
\cO'_K(n)=\sets{t\in K^n}{\nn t_1=\sum*_i\n{t_i}\le1},\qquad n\ge0.
\end{equation}
Indeed, for $s\in\cO'_K(m)$, $t\in\cO'_K(n)^m$, we have
$\nn{s(t)}_1=\sum_j \n{\sum_i s_it_{ij}}\le\sum_i\n s_i\sum_j\n{t_{ij}}\le\sum_i\n s_i\le1$, hence $s(t)\in\cO'_K(n)$.
In particular, $\cO'_K(1)=\cO_K$, hence $s(\cO_K^m)\sbs\cO_K$ and $s\in\bar\cO_K(m)$.
Therefore $\cO'_K\sbs\bar\cO_K\sbs\bar K$.
The monads $\cO'_K$ and $\bar\cO_K$ are different if the seminorm is non-archimedean, but they coincide for some archimedean valuations.

\begin{lemma}
If $K\sbs\bC$ is a semifield closed under conjugation, then
$\bar\cO_K=\cO_K'$.
\end{lemma}
\begin{proof}
We have seen that $\cO'_K\sbs\bar\cO_K$.
Conversely, let $t\in \bar\cO_K(n)$.
For $x\in K^n$ and $c\in\bQ$ satisfying
$c>\nn x_\infty$,
we have $x/c\in\cO^n_K$, hence $t(x/c)\in\cO_K$ and $\n{t(x)}\le c$.
This implies $\n {t(x)}\le \nn x_{\infty}$.
We choose $x\in K^n$ with $x_i=\bar t_i/(\n{t_i}+\eps_i)$, where $\n t_i+\eps_i\in\bQ$ and $\eps_i>0$.
Then $t(x)=\sum_i\frac{\n{t_i}^2}{\n {t_i}+\eps_i}\le1$ and we conclude that $\sum_i\n {t_i}\le1$.
\end{proof}

\begin{lemma}\label{lm:conv-satur}
Let $K\sbs\bC$ be a semifield such that $1-t\in K$ for all $t\in K\cap(0,1)$.
Let $X\sbs Y$ be a saturated submodule of $Y\in \Mod(\cO_K)$
and $x\in X$, $y\in Y$ be such that
$(1-u)x+uy\in X$ for some $u\in K\cap(0,1)$.
Then $(1-t)x+ty\in X$ for all $t\in K\cap[0,1)$.
\end{lemma}
\begin{proof}
Let $x_t=(1-t)x+ty$ for $t\in K\cap[0,1]$.
As $x_0=x\in X$ and $x_u\in X$, we have $x_t\in X$ for all $t\in[0,u]$.
Indeed, with $a=t/u\in K\cap[0,1]$, we have
$x_t=(1-au)x+auy=(1-a)x+ax_{u}\in X$.

For $t\in K\cap(0,1)$, we have
$x_t\sim (1-t)x_u+ty=(1-t)(1-u)x+(u-tu+t)y=x_{u(t)}$,
where $u(t)=(1-t)u+t\in(t,1)$.
An interval $[u,t']\sbs(0,1)$ can be covered by finitely many intervals of the form $[t,u(t)]$ for $t\in K\cap(0,1)$.
Assuming by induction that $x_t\in X$, we obtain $x_{u(t)}\in X$
as $x_{u(t)}\sim x_t$ and $X$ is saturated.
Therefore $x_s\in X$ for all $s\in K\cap[0,u(t)]$.
We conclude that that $x_{t'}\in X$.
\end{proof}

\begin{theorem}
If $K$ is a complete valued field, then $\Mod(\cO_K)$ is \qpa.
\end{theorem}
\begin{proof}
If the valuation is non-archimedean, then $\cO_K$ is a ring, hence $\Mod(\cO_K)$ is \qpa.
If the valuation is archimedean, Ostrowski's theorem
implies that there is an isomorphism $\si:K\to\bR$ or $\bC$ and a number $p\in(0,1]$
such that $\n a=\n{\si a}^p$ for all $a\in K$.
Then $\cO_K=\sets{a\in K}{\n{\si a}\le1}$
and we can assume that $K=\bR$ or $\bC$.

By Lemma \ref{para-crit}, we need to show that given submodules $X\sbs Y\sbs Z$ with saturated $Y\sbs Z$, the map $Y/X\to Z/X$ is injective.
Let us denote $\ang{X\xx X}_Y$ by $\sim$ and
$\ang{X\xx X}_Z$ by $\simm$.
Injectivity of $Y/X\to Z/X$ means that if $y\simm y'$ for some $y,y'\in Y$, then $y\sim y'$.

The equivalence relation $\sim$ is generated by \eqref{rel WPS}
\[\sum*_i s_ix_i+\sum*_jt_jy_j\sim \sum*_jt_jy_j\]
for $(s,t)\in \bar \cO_K(m+n)\sbs \cO_K^m\xx \cO_K^n$, $x\in X^m$ and $y\in Y^n$.
Assuming that $s,t\ne0$, let
$a=\nn s_1$, $b=\nn t_1$ and $u=\frac b{a+b}\in K\cap(0,1)$.
Then $s'=\frac1{1-u}s$ and $t'=\frac1{u}t$ satisfy
$\nn {s'}_1=a+b\le1$ and $\nn {t'}_1=a+b\le1$.
For $x'=s'(x)\in X$ and $y'=t'(y)\in Y$, we obtain
$s(x)+t(y)=(1-u)x'+u y'$.
Therefore the equivalence relation $\sim$ is generated by simple relations
\[(1-t)x+ty\sim ty\]
for $t\in K\cap[0,1)$, $x\in X$ and $y\in Y$.
We call $y$ the top element of this relation.

If $y\simm y'$ for some $y,y'\in Y$, they are connected by a chain of simple relations (with top elements in $Z$).
Since $Y$ is saturated in $Z$, all elements of the chain are contained in $Y$.
Therefore, we can assume that $y,y'$ are connected by a simple relation, hence $y=(1-t)x+tz$ and $y'=tz$ for some $t\in K\cap[0,1)$, $x\in X$ and $z\in Z$.
We can assume that $t\ne0$ as otherwise
$y=x\sim 0=y'$.

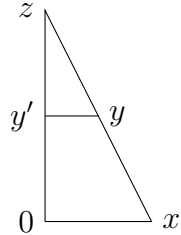
\[\tikz[scale=.7]{
\coordinate(0)at(0,0);
\coordinate(x)at(2,0);
\coordinate(z)at(0,4);
\coordinate(y)at(0,2);
\coordinate(y')at(1,2);
\draw(0)node[left]{0}--(x)node[right]{$x$}--
(z)node[left]{$z$}--(0);
\draw(y)node[left]{$y'$}--(y')node[right]{$y$};
}\]

We have $x\in X\sbs Y$ and $y=(1-t)x+tz\in Y$,
hence $(1-s)x+sz\in Y$ for all $s\in K\cap[0,1)$ by Lemma \ref{lm:conv-satur}.
Similarly, $sz\in Y$ for all $s\in K\cap[0,1)$.
Taking the convex hull of these lines, we conclude that all points of the above triangle except the point $z$ are contained in $Y$.
Now we can connect $y$ and $y'$ by a chain of simple relations with top elements in $Y$.
Therefore $y\sim y'$.
\end{proof}

\begin{remark}
The same argument applies to any semifield $K\sbs\bC$ such that $\n t\in K$ for $t\in K$ and $1-t\in K$ for $t\in K\cap(0,1)$.
Such semifield is closed under conjugation since $\bar t=\n t^2/t\in K$ for $0\ne t\in K$.
In particular, the argument applies to $K=\bR_+$ with $\cO_K=D_*=[0,1]$, $K=\bQ_+$ with $\cO_K=\bQ\cap [0,1]$,
%$K=\bR$ with $\cO_K=[-1,1]=\ZZ$,
and $K=\bQ$ with $\cO_K=\bQ\cap\ZZ$ (denoted by $\bZ_{(\infty)}$ in \dur).
\end{remark}

\begin{theorem}
The category $\Mod(D_*)$ of pointed convex spaces is \qpa.
\end{theorem}

\providecommand{\bysame}{\leavevmode\hbox to3em{\hrulefill}\thinspace}
\providecommand{\href}[2]{#2}

%\bibliography{papers}
%\bibliographystyle{hamsplain}
\end{document}